\documentclass{amsart}
\usepackage[T1]{fontenc} 

\theoremstyle{plain}

\usepackage[utf8]{inputenc}
\usepackage[english]{babel}
\usepackage{cite}
\usepackage{color}
\usepackage{hyperref}
\usepackage{float}
\usepackage{tikz}
\usetikzlibrary{automata}
\usepackage{mathtools}

\usepackage{amsthm}
 \newtheorem{theorem}{Theorem}[section]
\theoremstyle{definition}
\newtheorem{definition}[theorem]{Definition}
\newtheorem{lemma}[theorem]{Lemma}

\newcommand{\cay}{\operatorname{Cay}}

\newtheorem{claim}[theorem]{Claim}
  \newtheorem*{claim*}{Claim}

\newtheorem{example}[theorem]{Example}

\newtheorem{corollary}[theorem]{Corollary}

 \newtheorem{remark}[theorem]{Remark}
 
\newtheorem{proposition}[theorem]{Proposition}

\usepackage{amssymb}

\newcommand{\Pone}{\text{Cone}}
\tikzset{vertex/.style={circle, draw, fill=black!50},inner sep=0pt, minimum width=4pt}

\theoremstyle{definition}

\thanks {The first author was partially supported by NSF grant DMS-1607616. The second author is supported by NSF grant DMS-1812021.}

\begin{document}

\title{Regular Languages for Contracting Geodesics}
\author{Joshua Eike, Abdul Zalloum}

\begin{abstract}
Let $G$ be a finitely generated group. We show that for any finite symmetric generating set $A$, the language consisting of all geodesics in Cay$(G,A)$ with the contracting property is a regular language. An immediate consequence is that the existence of an infinite contracting geodesic in a Cayley graph of a finitely generated group implies the existence of a contracting element. In particular, torsion groups can't contain an infinite contracting geodesic. As an application, this implies that any finitely generated group containing an infinite contracting geodesic must be either virtually $\mathbb{Z}$ or acylindrically hyperbolic.
\end{abstract}

\maketitle

\section{Introduction}

The study of Gromov hyperbolic groups has been so fruitful that extending tools from this setting to more general classes of groups is a central theme in geometric group theory.

Among the fundamental tools is the study of the geodesic language in hyperbolic groups. A classic result by Cannon \cite{Cannon1984} shows that for any finitely generated hyperbolic group $G$, the language consisting of geodesic words in Cay$(G,S)$ is a regular language regardless of the chosen generating set $S$. A regular language is simply a (typically infinite) set of words of low enough complexity that it can be produced by a finite graph. The existence of such a language has beautiful geometric, algebraic, analytical, and combinatorial consequences. For instance, a geometric reflection of the existence of a regular language for geodesics in the group is the finiteness of cone types in hyperbolic groups. This in turns implies the algebraic fact that the word problem is solvable \cite{Word_Processing}.

The existence of a regular language also has the analytical/combinatorial consequence that the growth function of the group is a linear recursive function. A beautiful, mind boggling example was given by Cannon of a finitely presented group where the language of all geodesic words is a regular language with respect to one generating set but not another. This is Example 4.4.2 in \cite{Word_Processing}, and the group is $\mathbb{Z}\rtimes\mathbb{Z}_2 = \langle x,y,z \mid x^2,zyz^{-1}y^{-1},yxz^{-1}x^{-1}\rangle$. Therefore, the regularity of the geodesic language in a finitely presented group $G$ is not an intrinsic property of the group and is sensitive to the chosen generating set. As the example $\mathbb{Z}\rtimes\mathbb{Z}_2$ above shows, regularity of the geodesic language is not independent of the presentation even in virtually abelian groups.

We show that in a finitely generated group $G$, if one restricts their attention to the language consisting of all ``hyperbolic-like" geodesics, one gets a regular language for any generating set. The condition on a geodesic $\alpha$ to be ``hyperbolic-like" is that for any subsegment $\beta \subseteq \alpha$, the projection to $\beta$ of any ball disjoint from $\beta$ has diameter bounded above by $D$. A geodesic satisfying this condition is called \emph{$D$-super-contracting}.

A related well-studied "hyperbolic-like" property is that of a contracting geodesic. A geodesic $\alpha$ is said to be $\emph{D-contracting}$ if projections to $\alpha$ of balls disjoint from it have diameter at most $D$. In the literature, the exact value of $D$ is seldom important. What is interesting is that such a parameter exists and is independent of the length of the geodesic. This means that projections onto such a geodesic are similar to projections in negatively-curved geometry and different from projections in flat geometry. While our notion of $D$-super-contracting is technically stronger than $D$-contracting, the difference is only in the parameterization and completely independent of the length of the geodesic. In fact, we show the following.

\begin{theorem} \label{thm: 2nd intro thm}
A geodesic $\alpha$ in a proper geodesic metric space $X$ is $D$-contracting if and only if it is $D'$-super-contracting where $D$ and $D'$ determine each other.
\end{theorem}

In light of the above theorem, the notions of super-contracting and contracting geodesics are equivalent, up to changing the contraction's parameter. Our main theorem is the following.

\begin{theorem} \label{thm: first intro thm}

Let $G$ be a finitely generated group, and let $A$ be any finite symmetric generating set. Then the language $L_D$ consisting of all $D$-super-contracting geodesic words in Cay$(G,A)$ is a regular language for any $D$.

\end{theorem}

Since a finitely generated group $G=\langle A \rangle$ is hyperbolic if and only if every geodesic in Cay($G,A)$ is $D$-super-contracting for a uniform $D$ (see Remark \ref{rmk: contraction characterization for hyperbolicity}), the above theorem recovers a classic result by Cannon where he shows that for a hyperbolic group $G$, and for a finitely generated set $A$, the language consisting of all geodesics in Cay$(G,A)$ is a regular language. Since the generating function counting the number of words of length $n$ in a regular language is always rational, this opens a host of combinatorial questions. For any finitely generated group, choice of generating set, and parameter $D$, we can ask how many geodesics of length $n$ are $D$-super-contracting. All of these questions can be answered with a rational generating function. An immediate consequence to Theorem \ref{thm: 2nd intro thm} and Theorem \ref{thm: first intro thm} is the following.

\begin{corollary} \label{cor: first cor intro}
Let $G$ be a finitely generated group and let $A$ be a finite symmetric generating set for $G.$ If Cay$(G,A)$ contains an infinite contracting geodesic, then $G$ contains a contracting isometry. In particular, torsion groups do not contain infinite contracting geodesics.
\end{corollary}

As an interesting application to the above theorem, we answer the following question posed by Osin: Does the existence of an infinite contracting geodesic in a finitely generated group imply that the group is acylindrically hyperbolic?

The previous theorem yields a positive answer to this question:

\begin{corollary} \label{cor: 2nd cor intro}
Let $G$ be a finitely generated group with a generating set $A$ such that Cay$(G,A)$ contains an infinite contracting geodesic, then $G$ must be acylindrically hyperbolic.
\end{corollary}

The paper is organized as follows:

In Section \ref{sec: introducing super contraction}, we introduce an a priori stronger notion of a contracting geodesic that we call \emph{super-contracting} and we set up the main tools needed for the proof of the main two theorems.

In Section \ref{sec: main theorem}, we prove Theorem \ref{thm: first intro thm}.

In Section \ref{sec: super contracting is equivalent to contracting}, we prove Theorem \ref{thm: 2nd intro thm}, Corollary \ref{cor: first cor intro} and Corollary \ref{cor: 2nd cor intro}.



\subsection*{Acknowledgement}
The first author would like to thank his advisor Ruth Charney for her steadfast support and guidance. The second author would like to thank his advisers Johanna Mangahas and Ruth Charney for their exceptional support and guidance. He would also like to thank David Cohen and Denis Osin for fruitful discussions. Both authors would like to thank Kim Ruane for helpful conversation. They also would like to thank the referee(s) for their helpful suggestions.

\section{super-contracting geodesics} \label{sec: introducing super contraction}

Throughout this section $X$ denotes a proper geodesic metric space.  Also, if $\alpha$ is a path in $X$, we will generally use $\alpha$ to denote the image of $\alpha$ in the space $X.$

\begin{definition}[projection]\label{projection}
Let $C$ be a closed subset of $X$. We define the \emph{projection} of a point $x$ onto $C$ to be
\[\pi_C(x) = \{p\in C \mid d(x,p) = \min_{y\in C}d(x,y) \}. \]
In general $\pi_C(x)$ may contain more than one point. For $B\subset X$ we define the \emph{projection} to be $\pi_C(B) = \cup_{x\in B} \pi_C(x)$. We write $d(x, C)$ to denote the distance from $x$ to its projection points.
\end{definition}

\begin{definition}[contracting]\label{contracting qgeod}
Let $\alpha$ be a continuous quasi-geodesic (possibly infinite). We say that $\alpha$ is \emph{$D$-contracting} if for any closed metric ball $B$ disjoint from $\alpha$, diam($\pi_{\alpha}(B)) \leq D$. We say it is \emph{contracting} if it is $D$-contracting for some $D$. 

\end{definition}
\begin{definition}[super-contracting] \label{def: super-contracting}
Let $\alpha$ be a continuous quasi-geodesic in $X$. We say that $\alpha$ is \emph{$D$-super-contracting} if every sub-segment of $\alpha$ is $D$-contracting in the above sense. That is, if for any subsegment $\gamma \subseteq \alpha$ and any closed metric ball $B$ disjoint from $\gamma$, $\text{diam}(\pi_\gamma(B))\leq D$. We say $\alpha$ is \emph{super-contracting} if it is $D$-super-contracting for some $D$. It is immediate from the definition of a $D$-super-contracting quasi-geodesic $\alpha$ that a subsegment of $\alpha$ is itself $D$-super-contracting.

\end{definition}
It is clear from the above definitions that if a geodesic is $D$-super-contracting, then it is $D$-contracting. As mentioned in the table of contents, Section \ref{sec: super contracting is equivalent to contracting} is devoted to proving a converse of the above statement. In other words, we prove that for any proper geodesic metric space $X$ and for any $D$-contracting geodesic $\alpha \subseteq X$, there exists $D'$, depending only on $D$, such that $\alpha$ is $D'$-super-contracting. Also, notice that projections on contracting/super-contracting quasi-geodesics are coarsely well defined. More precisely, if $\alpha$ is a $D$-contracting geodesic, then $ \forall x \in X$ we must have $\text{diam}(\pi_\alpha(x))\leq D$.



\begin{definition}[Quasi Isometric embedding] \label{Def:Quasi-Isometry} 
Let $(X , d_X)$ and $(Y , d_Y)$ be metric spaces. For constants $\lambda \geq 1$ and
$\epsilon \geq 0$, we say a map $f: X \to Y$ is a 
$(\lambda, \epsilon)$--\textit{quasi-isometric embedding} if, for all points $x_1, x_2 \in X$
$$
\frac{1}{\lambda} d_X (x_1, x_2) - \epsilon  \leq d_Y \big(f (x_1), f (x_2)\big) 
   \leq \lambda \, d_X (x_1, x_2) + \epsilon.
$$
If, in addition, there exists a constant $C \geq 0$ such that every point in $Y$ lies in the $C$--neighbourhood of the image of 
$f$, then $f$ is called a $(\lambda, \epsilon)$--\emph{quasi-isometry}. When such a map exists, $X$ 
and $Y$ are said to be \textit{quasi-isometric}. 

\end{definition}

\begin{definition}[Quasi-geodesics] \label{Def:Quadi-Geodesic} 
A $(\lambda, \epsilon)$-\emph{quasi-geodesic} is a $(\lambda, \epsilon)$-quasi-isometric 
embedding $\gamma: [a, b] \to X$. A quasi-geodesic is a  $(\lambda, \epsilon)$-\emph{quasi-geodesic} for some $\lambda \geq 1,$ and $\epsilon \geq 0.$
\end{definition}

\begin{definition}[Morse] A quasi-geodesic $\gamma$ in a proper geodesic metric space
is called \emph{$N$-Morse}, where $N$ is a function $N:[1, \infty) \times [0, \infty) \rightarrow [0, \infty)$, if for any
$(\lambda, \epsilon)$-quasi-geodesic $\sigma$ with endpoints on $\gamma$, we have $\sigma \subseteq \mathcal{N}_{N(\lambda,\epsilon)}(\gamma)$. The
function $N(\lambda, \epsilon)$ is called a \emph{Morse gauge}.
\end{definition}

The following is Lemma 3.3 in \cite{Sultan2013}.
\begin{lemma} \label{lem: contracting implies Morse}(Contracting implies Morse)
For any proper geodesic metric space $X,$ and for each $D \geq 0,$ there exists an $N$, depending only on $D$ such that every $D$-contracting geodesic is $N$-Morse. In particular, every $D$-super-contracting geodesic is $N$-Morse.

\end{lemma}

\begin{lemma}\label{lem: Morse is contagious}
Let $M$ be a Morse gauge and let $\alpha$ be an $M$-Morse geodesic in Cay$(G,A)$ starting at the identity. If $\gamma$ is any geodesic in the Cay$(G,A)$ starting at the identity and ending 1 apart from $\alpha$ then $\gamma$ is $N$-Morse where $N$ depends only on $M$. Also, $N\geq M$.
\end{lemma}

\begin{proof}
The proof of this lemma follows easily from Lemma 2.1 in \cite{Cordes2017}.
\end{proof}

The following is Lemma 2.7 in \cite{Cordes2017}. It basically states that if you have two $N$-Morse geodesics with the same origin that end close to each other then they have to be roughly uniformly close.

\begin{lemma}\label{lem: fellow traveling}
If $\alpha_1, \alpha_2:[0,A] \rightarrow X$ are $N$-Morse geodesics with $\alpha_1(0) = \alpha_2(0)$ and $d(\alpha_1(s), \alpha_2)\leq K$ for some $ s \in [0,A]$ and some $K>0$, then $d(\alpha_1(t), \alpha_2(t))\leq8N(3,0)$ for all $t<s-K-4N(3,0)$.
\end{lemma}

\begin{lemma} \label{lem: ending close to contracting implies fellow travelling}
 If $\alpha$ is a $D$-super-contracting geodesic in Cay$(G,A)$ and $\beta$ is another geodesic with the same starting point as $\alpha$ and ending at most 1 apart from where $\alpha$ ends, then there is a constant $C\geq1$ only depending on $D$ such that $d(\alpha(t), \beta(t)) \leq C$ for all $t< |\alpha|-1-\frac{C}{2}$.
\end{lemma}

\begin{proof}

A $D$-super-contracting geodesic is $D$-contracting and therefore $M$-Morse by Lemma \ref{lem: contracting implies Morse}. Combining this with Lemmas \ref{lem: Morse is contagious} and \ref{lem: fellow traveling} we get the following. If $\alpha$ is any $D$-super-contracting geodesic in Cay$(G,A)$ starting at the identity, then $\alpha$ has to be $M$-Morse where $M$ depends only on $D$. 
By Lemma \ref{lem: Morse is contagious}, if $\beta$ is any other geodesic in Cay$(G,A)$, starting at the identity and ending 1 apart from where $\alpha$ ends, then $\beta$ has to be $N$-Morse where $N$ depends only on $M$ which depends only on $D$.
Now Lemma \ref{lem: fellow traveling} gives us that $d(\alpha (t), \beta(t)) \leq 8N(3,0)$ for all $t<|\alpha|-1-4N(3,0)$. Thus, the conclusion follows by taking $C=8N(3,0).$

\end{proof}

The following lemma relates projection onto a subsegment of a super-contracting geodesic to projection onto the full geodesic. It shows that if the projection onto the full geodesic is on one side of the subsegment, then projection onto the subsegment can't jump too far toward the other side.

\begin{lemma}[bounded jumps]\label{lem: Gap}
Let $\gamma:[a, b]\to X$ be a $D$-super-contracting geodesic. Let $c\in[a,b]$. Denote by $\alpha$ and $\beta$ the subsegments of $\gamma$ from $\gamma(a)$ to $\gamma(c)$ and from $\gamma(c)$ to $\gamma(b)$ respectively.
For any $x\in X$, if $\pi_\gamma(x) \cap \alpha \neq \emptyset$, then every $\gamma(q)\in\pi_\beta(x)$ satisfies $d(\gamma(c),\gamma(q)) \leq D$. 
\end{lemma}

\begin{proof}Let $x\in X$ be such that there exists a point $\gamma(p)\in\pi_\gamma(x)\cap\alpha\neq\emptyset$. Since $\gamma$ is continuous, the map $f(t)=d(x, \gamma(t))$ is a continuous real-valued function. Let $\gamma(q)$ be in $\pi_\beta(x)$ and let $\gamma(q')$ be the point in $\alpha$ closest to $\gamma(c)$ satisfying $d(x,\gamma(q'))=d(x,\gamma(q))$. That is, let $q' = \sup\{s\leq c \mid f(s) = f(q)\}$. Such a point exists by the intermediate value theorem applied to the inequality $f(p)\leq f(q)\leq f(c)$. Taking the supremum guarantees that no point between $\gamma(q')$ and $\gamma(c)$ is closer to $x$ than $\gamma(q')$ is. Thus, $\gamma(q)$ and $\gamma(q')$ are both in the projection of $x$ onto $\gamma|_{[q',q]}$. Since $\gamma$ is $D$-super-contracting, $d(\gamma(q'),\gamma(q))\leq D$. Therefore $d(\gamma(c),\gamma(q))\leq D$.
\end{proof}


\section{super-contracting language is a regular language in any finitely generated group}\label{sec: main theorem}

\begin{definition}[A language over a finite alphabet]
Let $A$ be a finite set and $A^\star$ be the free monoid over $A$. We call an element $w \in A^\star$ a \emph{word} in $A$. If $w \in A^\star$ and $w = a_1 \cdots a_n$ where each $a_i \in A$, then we call $a_1,\dots,a_n$ the \emph{letters} of the word $w$.  A \emph{language over $A$} is a set of words in $A^\star$.  The \emph{word length} of  $w \in A^\star$ is the number of letters of $A$ in the word $w$. We denote the word length of $w$ by $|w|$.
\end{definition}

The alphabets for the languages we will be working with will be  finite generating sets for groups.

\begin{definition}
Let $G$ be a finitely generated group and let $A$ be a finite, symmetric generating set for $G$. The \emph{Cayley graph of $G$ with respect to $A$}, $\cay(G,A)$, is the graph whose vertices are the elements of $G$ and $g,h \in G$ are joined  by an edge if $g^{-1}h \in A$. If $g^{-1}h \in A$, then we label the edge connecting $g$ and $h$ by $g^{-1}h$.  The Cayley graph $\cay(G,A)$ is a metric space by declaring each edge to have length $1$. 
\end{definition}

\begin{definition} (Left actions on Cayley graphs) Let $G$ be a finitely generated group and let $A$ be a finite symmetric generating set. The group $G$ admits a natural left action on $\cay(G,A)$ by isometries as follows. Let $g\in G$ and $x\in \cay(G,A)$. If $x$ is a vertex, then it is also a group element, so we can define $g \cdot x = gx$ where $gx$ is simply the product of $g$ and $x$ in $G$. If $x$ is a point on an edge connecting two vertices $v_1$ and $v_2$ at distance $t\in [0,1]$ from $v_1$, let us denote it $x = (v_1, t, v_2)$. We define the group action on such a point to be $g\cdot x = (gv_1, t, gv_2)$. For a subset $B \subseteq \cay(G,A),$ we define $g\cdot B:=\{g\cdot x\,| x \in B\}$. For a continuous path $\sigma: [a,b] \rightarrow \cay(G,A)$, we define a new path $g \cdot \sigma:[a,b] \rightarrow \cay (G,A)$ by $(g\cdot \sigma)(t)=g\cdot \sigma(t).$ We will sometimes abuse notation and use  $g \cdot \sigma $ to denote the image of the map $g\cdot \sigma.$
 
\end{definition}

When the group $G$ is generated by the finite set $A$, every path in $\cay(G,A)$ produces a word in $A^\star$ by concatenating the labels of the edges in the order they appear along the path. Conversely, every word in $A^\star$ produces a path in $\cay(G,A)$ by starting at the identity $e$ and traversing the edges in $\cay(G,A)$ labeled by the letters of the word appearing from left to right. We will be particularly concerned with words in $A^\star$ that correspond with geodesic paths in $\cay(G,A)$.

\begin{definition}
For a word $w \in A^\star$, let $\overline{w}$ denote the element of $G$ obtained by viewing $w$ as an element of $G$.
A word $w \in A^\star$ is \emph{geodesic} in $\cay(G,A)$ if the path in $\cay(G,A)$ from $e$ to $\overline{w}$ labeled by the letters of $w$ is a geodesic in $\cay(G,A)$. This is equivalent to saying that $|w|$ is minimal among all words that represent $\overline{w}$. Further, for a word $w \in A ^\star,$ we define $|\overline{w}|_G$ to be the minimal word length among all possible words $u \in A^\star$ with $\overline{u}=\overline{w}.$ In notation, we have 

$$|\overline{w}|_G:=\text{min}\{|u|\,\, \text{for all}\,\, u \in A^\star \text{such that}\,\, \overline{u}=\overline{w}\}.$$
\end{definition}

\begin{definition}\label{def: fsa}
A \emph{finite state automaton (FSA)} over an alphabet $A$ is a finite graph whose edges are directed and
labeled by elements of $A$; the vertices of the graph are divided into two sets—“accept”
and “reject”—and there is a distinguished vertex $s_0$ called the initial vertex. The \emph{accepted language} of the automaton is the set
of words which occur as labels on a directed edge path beginning at $s_0$ and ending
at an accept vertex. Those vertices are often called the \emph{states} of the automaton.
\end{definition}

\begin{definition}\label{def: regular language} (Regular languages)
Let $A$ be a finite set and $A^*$ be the set of all words with letters in $A$. Recall that a language over $A$ is a subset $L\subset A^*$. A language over $A$ is \emph{regular} if it is the accepted language of some finite state automaton over $A$.
\end{definition}

\begin{definition}
Let $u \in A^\star$ be a geodesic word in the Cay($G,A)$. We define the \emph{$D$-super-contracting cone} of $u$, denoted by $\Pone^{D}(u)$, to be all $w \in A^\star$ so that the concatenation $uw$ is a $D$-super-contracting geodesic in Cay$(G,A)$. If $u$ and $v$ have the same $D$-super-contracting cone, we will say that $u$ and $v$ have the same \emph{cone type}.
\end{definition}

\begin{example}\label{example: D-cone}

In the free group on two letters $F_2=\langle a,b \rangle $, if we take $D=1$, the $D$-super-contracting cone of $a$ is all geodesic words in the group that don't start with $a^{-1}$ whereas in $\mathbb{Z}\oplus \mathbb{Z}=\langle a,b|[a,b] \rangle$, if $D=1$, the $D$-super-contracting cone of $a$ is empty. If we take $D=2,$ then in the free group example, the $D$-super-contracting cone of $a$ will still be all geodesic words in the group that don't start with $a^{-1}$ whereas in the $\mathbb{Z}\oplus \mathbb{Z}$ example the $D$-super-contracting cone of $a$ is now the set $\{a,b,b^{-1}\}.$
\end{example}

\begin{definition}\label{def: Cay k-tail}
Given an element $u$ in $A^\star$ representing a geodesic in Cay$(G,A)$, let $\overline{u}$ be the unique group element represented by the word $u$. Given $k \in \mathbb{N}$, we define the \emph{$k$-tail} of $u$ to be all elements $h \in G$ with $|h|_G\leq k$ such that $|\overline{u}h|_G<|\overline{u}|_G$. We denote the $k$-tail of $u$ by $T_k(u)$.
\end{definition}

The $k$-tail of a geodesic word $u$ is all group elements in a ball of radius $k$ of the identity that move $\overline{u}$ closer to the identity in the Cayley graph.

\begin{definition}\label{def: Cay contracting type} Given $u \in A^\star$ representing a geodesic in the Cay$(G,A)$, $m \in \mathbb{N}$, and $D\geq0$, we define the \emph{$m$-local $D$-contracting type} of $u$ to be all words $w \in A^\star$ with $|w|\leq m$ such that the concatenation $uw$ is a $D$-super-contracting geodesic in Cay$(G,A)$. We will denote the $m$-local $D$-contracting type of a geodesic word $u$ by $\Pone^D_m(u)$.
\end{definition}

The following lemma is a slight modification of the Lemma \ref{lem: Gap}. It will be used in the proof of Theorem \ref{General Result}.

\begin{lemma} \label{lemma: bounded jumps even wthout full contraction}
Let $G$ be a group with a finite symmetric generating set $A$. Let $u$ and $w$ be words in $A$ and $a \in A$. Suppose that $uwa$ is a geodesic word such that $|w| \geq 3D+1$ and $uw$ is $D$-super-contracting. For any vertex $y \in $Cay$(G,A)$, if $\pi_{uwa}(y) \cap u \neq \emptyset,$ then $d(\overline{u},p) \leq D$ for any $p \in \pi_{wa}(y)$.
\end{lemma}

\begin{proof} Let $X=$Cay$(G,A)$, notice that since $X$ is a graph, for any geodesic word $v$ in $X$ and any vertex $x \in X,$ we have $d(x, \pi_v(x)) \in \mathbb{N} \cup \{0\}.$ In other words, there exists at least one vertex $x' \in v$ such that $d(x,x')=d(x,\pi_v(x)).$ That is to say, as $X$ is a graph, no projection of $x$ to $v$ can lie on an interior of an edge in $v.$ 


 \begin{claim} \label{claim: comparing projections} Let $uwa$ and $y$ be as in the statement of the lemma, we have the following.
 
 \begin{enumerate}
     \item  All projections of $y$ to $uwa$ live on the subsegment $uw.$ That is, $\pi_{uwa}(y)=\pi_{uw}(y).$

     \item All projections of $y$ to $wa$ live on the subsegment $w.$ That is, $\pi_{wa}(y)=\pi_{w}(y).$
 \end{enumerate}

\end{claim}

Before proving the above claims, we show how they imply the statement of the lemma. Notice that by the assumption of the lemma, there exists a point $p_y \in \pi_{uwa}(y) \cap u.$ As $uw$ is a subsegment of $uwa,$ we have $p_y \in \pi_{uw}(y) \cap u.$ Since $uw$ is $D$-super-contracting, using Lemma \ref{lem: Gap}, all projections of $y$ to $w$ must be within $D$ of $\overline{u}.$ That is to say, for every $p \in \pi_{w}(y),$ we have $d(p,\overline{u}) \leq D$. However, using part (2) of Claim \ref{claim: comparing projections}, we have $\pi_{wa}(y)=\pi_{w}(y)$ which finishes the proof.

\vspace{2mm}

\emph{Proof of }(1):
Let $p_y \in \pi_{uwa}(y) \cap u$ and let $d=d(y, p_y)=d(y, \pi_{uwa}(y)).$ In order to prove the claim, it suffices to show that $d(y, \overline{uwa}) > d.$ Suppose not, that is, suppose $d(y,\overline{uwa)} \leq d$. As $d=d(y, \pi_{uwa}(y)),$ we have $d(y,\overline{uwa)} = d.$ See Figure \ref{fig: bounded jump without full contraction}. Notice that $ d \leq d(y,\overline{uw}) \leq d+1$. However, $d(y,\overline{uw}) \neq d$ as this would imply diam$(\pi_{uw}(y)) \geq 3D+1$ contradicting the assumption that $uw$ is $D$-super-contracting. Hence, we have $d(y,\overline{uw})= d+1$. Now, notice that since $uw$ is a subsegment of $uwa,$ the point $p_y$ must also live in $\pi_{uw}(y).$ Let $p_y'$ be the unique point in $\pi_{uw}(y)$ so that $d(e,p_y')=\underset{p \in \pi_{uw}(y)}{\text{max}} d(e,p)$. We claim that $p_y'$ lives on the $w$ subsegment of $uw$. Suppose not. That is, suppose $p_y' \in uw$ and $p_y' \notin w$. This implies that $d(y, p_y')=d(y, \pi_{uw}(y)) \geq d$ and thus every point in $w$ is at least $d+1$ away from $y.$ Consequently, we have $\overline{uw} \in \pi_w(y)$ contradicting Lemma \ref{lem: Gap} since $|w| \geq 3D+1$. Hence, we have $p_y' \in w.$ Let $q$ be the unique vertex on $uw$ so that $d(e,q)=d(e,p_y')+1$. Notice that since $uw$ is $D$-super-contracting, we have $d(q,p_y) \leq d(q,p_y')+d(p_y',p_y) \leq 1+D.$ Furthermore, by our choice of $q$, we have $d(y,q)=d+1$ and $d(y,q') \geq d+1$ for any vertex $q'$ on the geodesic $uw$ with $d(e,q') \geq d(e,q).$ Now, let $\alpha=[q,\overline{uw}]_{uw}$ denote the subsegment of $uw$ starting at the vertex $q$ and ending at the vertex $\overline{uw}$. The points $q,\overline{uw}$ both live in $\pi_\alpha (y)$ as $d(y,q)=d(y,\overline{uw})=d+1$ where $d(y,z) \geq d+1$ for any vertex $z \in \alpha$. On the other hand, $d(q,\overline{u})=1+d(p_y',\overline{u}) \leq 1+d(p_y',p_y) \leq 1+D$. This yields that $d(q,\overline{uw})=d(\overline{u},\overline{uw})-d(q,\overline{u}) \geq (3D+1)-(D+1)=2D,$ which contradicts the fact that $uw$ is $D$-super-contracting. Therefore, the projection of $y$ to $uwa$ is contained entirely in the subsegment $uw.$ Hence $\pi_{uw}(y)=\pi_{uwa}(y).$ \\

\begin{figure}
    \centering

\begin{tikzpicture}

\draw[ very thick,black] (0,0) -- ++(3,0) node[below] {};

\node[right] at (2.9,-.2) {$\overline{u}$};

\draw[ very thick,black] (6.5,0) -- ++(.4,.4);

\node[right] at (6.7,.1) {$a$};

\draw[thick,fill=black] (2,3.5) circle (0.05cm);

\node[above] at (2,3.6) {$y$};

\draw[dashed] (2,3.5) -- (2,0);

\draw[thick,fill=black] (2,0) circle (0.03cm);

\node[below] at (2.1,-.1) {$p_y$};

\node[below] at (3.8,-.1) {$p'_y$};

\draw[thick,fill=black] (3.8,0) circle (0.03cm);

\node[left] at (1.9,1.7) {$d$};

\draw[dashed] (2,3.5) -- (6.9,.4);

\node[above] at (5,1.7) {$d$};

\draw[thick,fill=black] (3,0) circle (0.05cm);
\draw[thick,fill=black] (6.5,0) circle (0.05cm);

\draw[thick,fill=black] (6.9,.4) circle (0.05cm);

\node[above] at (6.9,.52) {$\overline{uwa}$};

\node[below] at (1.4,-.1) {$u$};

\draw[ very thick,red] (3,0) -- ++(3.5,0) node[below]{};

\node[below] at (6.3,-.1) {$\overline{uw}$};

\node[below] at (4.7,-.1) {$w$};

\draw[thick,fill=black] (0,0) circle (0.05cm);
\draw[thick,fill=black] (0,0) circle (0.05cm);

\end{tikzpicture}
\caption{Possible projections of $y$ to the geodesic $uwa$ where $uwa$ is $D$-super-contracting and $|w| \geq 3D+1$ .}
    \label{fig: bounded jump without full contraction}

\end{figure}

\emph{Proof of }(2): Let $p_y \in \pi_{uwa}(y) \cap u$ and let $d=d(y, p_y)=d(y, \pi_{uwa}(y)).$ Using part (1), we have $d(y,\overline{uwa})>d.$ Notice that if $\pi_{uwa}(y) \cap w \neq \emptyset,$ then we are done as if $z \in \pi_{uwa}(y) \cap w,$ then $d(y,z)=d$ while $d(y,\overline{uwa})>d.$ Hence $\pi_w(y)=\pi_{wa}(y).$ 

Otherwise, if $\pi_{uwa}(y) \cap w = \emptyset,$ then all projections of $y$ to $uwa$ live in the $u$ subsegment of $uwa.$ In particular, since $uw$ is a subword of $uwa,$ all projections of $y$ to $uw$ live on the $u$ subsegment of $uw.$ Let $p''$ be the point with $d(\overline{u},p'')=\underset{p \in \pi_{w}(y)}{\text{max}} d(\overline{u},p)$. Notice that by Lemma \ref{lem: Gap}, every projection of $y$ to $w$ is at most $D$ away from $\overline{u}$, and since $|w| \geq 3D+1$, we have $p'' \neq  \overline {uw}.$ Let $q''$ be the point on $w$ with $d(\overline{u},q'')=d(\overline{u},p'')+1$. By Lemma \ref{lem: Gap}, since $uw$ is $D$-super-contracting, we have $d(\overline{u},p'') \leq D$ and hence $d(\overline{u},q'') \leq D+1.$ Notice that by our choice of $p''$ if we let $d'=d(y,p'')=d(y, \pi_w(y))$, then $d(y,q'')=d'+1$. In order to finish the proof of this claim, we need to show that $d(y,\overline{uwa}) \neq d'.$ Suppose for the sake of contradiction that $d(y,\overline{uwa})=d',$ this implies that $d(y,\overline{uw})$ is either $d'$ or $d'+1$. However, it can't be $d'$ as this would contradict the assumption that $uw$ is $D$-super-contracting. Thus $d(y, \overline{uw})=d'+1.$ But that implies that both $q''$ and $\overline{uw}$ live on the projection of $y$ to the subsegment of $uw$ connecting $q''$ to $\overline{uw}$ given by $\beta=[q'',\overline{uw}]_{uw}$. This contradicts the assumption that $uw$ is $D$-super-contracting as diam$(\pi_\beta(y))=d( \overline{uw}, q'')=d(\overline{uw},\overline{u})-d(\overline{u},q'') \geq (3D+1)-(D+1)=2D.$



\end{proof}

Let $G$ be a group with a finite symmetric generating set $A$. The goal of this section is to show that the languages $L_D$ consisting of all $D$-super-contracting geodesics are all regular languages regardless of the chosen generating set $A$. Before doing so, we will state an important key theorem that will provide us with the states needed for our FSA, we will refer to this theorem by ``the cone types theorem''. The theorem states that in order to determine $D$-super-contracting cone type of a geodesic word $u$, you need only to understand the local geometry around the vertex $\overline{u}$. To be more precise, it says that there exist a uniform $m$, depending only on $D$, such that the $m$-neighborhood around a vertex $\overline{u}$ encodes the information needed to determine what elements are in $\Pone^{D}(u)$. This will imply that we have only finitely many cone types because there are only finitely many types of $m$-neighborhoods in Cay$(G,A).$

\begin{theorem}\label{General Result}
For a given constant $D \geq 0,$ there exists an integer $m$, depending only on $D$, such that if $u,v$ are two geodesic words in Cay$(G,A)$ with $T_m(u)=T_m(v)$ and $\Pone^{D}_m(u)=\Pone^{D}_m(v)$, then $\Pone^{D}(u)=\Pone^{D}(v).$ In particular, there are only finitely many such cones.
\end{theorem}

We will prove this theorem, but for now, let us show how it implies our first main theorem about the existence of a regular language for all $D$-super-contracting geodesics:

\begin{theorem}\label{thm: regular language general}
Let $G$ be a group and $A$ any finite symmetric generating set. Let $L_D$ be the language of words in the alphabet $A$ which, when interpreted as paths in the Cayley graph, are $D$-super-contracting geodesics. Then for any fixed $D$, the language $L_D$ is a regular language.
\end{theorem}

\begin{proof}
Consider the finite graph $\Gamma$ whose vertices are the $D$-contracting cone types of $G$ and which has a directed edge labeled $a\in A\cup A^{-1}$ connecting the $D$-super-contracting cone type of a geodesic word $u$ to the $D$-super-contracting cone type of $ua$ if and
only if $a$ belongs to the $D$-super-contracting cone of $u$. Otherwise the edge labeled $a$ goes to the unique fail state. All non-empty $D$-super-contracting cone types are accept states. The initial state is the cone type of the empty word. The previous theorem shows that there are finitely many vertices, and the $D$-super-contracting cone type was defined precisely to pick out those continuations that are both geodesic and $D$-super-contracting. Also, observe that if $\Pone^D(u)=\Pone^D(v)$, then we have $\Pone^D(ua)=\Pone^D(va)$ for any $a \in A.$ To see this, first, notice that since $\Pone^D(u)=\Pone^D(v)$, the word $ua$ is a $D$-super-contracting geodesic if and only if $va$ is. Further, for any word $w,$ if we let $w'=aw$, we have $w' \in \Pone^D(u)$ if and only if $w' \in \Pone^D(v).$ This implies that $uw'=uaw$ is a $D$-super-contracting geodesic if and only if $vw'=vaw$ is. Therefore, $\Pone^D(ua)=\Pone^D(va)$ whenever $\Pone^D(u)=\Pone^D(v).$ This shows that it is not possible for a single vertex of the directed graph $\Gamma$ to have two distinct outgoing edges labelled with the same letter $a \in A$. Such a finite state automaton is called \emph{deterministic}.
\end{proof}

Now we prove Theorem \ref{General Result}:

\begin{proof} Fix a super-contracting constant $D$. Using Lemma \ref{lem: ending close to contracting implies fellow travelling}, there exists a constant $C$, which depends only on $D$, such that for any $D$-super-contracting geodesic $\alpha$, if $\beta$ is a geodesic starting at the same point as $\alpha$ and ending at most 1 away from $\alpha$, then $d(\alpha(t), \beta(t)) \leq C$ for all $t< |\alpha|-1-\frac{C}{2}$.

We want to show that there exists some $m$ large enough so that if two geodesic words $u,v \in$Cay$(G,A)$ satisfy $T_m(u)=T_m(v)$ and $\Pone^{D}_m(u)=\Pone^{D}_m(v)$, then $\Pone^{D}(u)=\Pone^{D}(v)$. Recall that for a word $w$ in $A^\star$, the number of letters appearing in $w$ is denoted by $|w|$ while the length of the group element $\overline{w}$ is denoted by $|\overline{w}|_G.$

Let $\overline{u},\overline{v}$ be the unique group elements represented by $u$ and $v$ respectively. Note that since $u,v$ are assumed to be geodesic words, we must have $|\overline{u}|_G=|u|$ and $|\overline{v}|_G=|v|$. Choose  $m>\text{max}\{C+1, 3D+1\}$.

We proceed by induction on the length of the words in the cone. Since $u$ and $v$ satisfy $T_m(u)=T_m(v)$ and $\Pone^{D}_m(u)=\Pone^{D}_m(v)$, if $w \in A^\star$ with $|w|\leq m$, then $w \in \Pone^{D}(u)$ if and only if $w \in \Pone^{D}(v)$. This covers the base cases. For the induction step, let $w \in \Pone^{D}(u) \cap \Pone^{D}(v)$ and consider $wa$ for $a\in A$. We want to show that $wa \in \Pone^{D}(u)$ if and only if $wa \in \Pone^{D}(v).$

Suppose for contradiction that $wa \in \text{Cone}^{D}(v)$ but $wa \notin \text{Cone}^{D}(u)$. The cones agree for words of length at most $m$, so $|wa| \geq m+1$. Since $wa \notin \Pone^{D}(u)$, then by definition, either the word $uwa$ is not a geodesic word or $uwa$ is a geodesic word that is not $D$-super-contracting.

First we show that $uwa$ must be a geodesic word. We remark that this part of the proof closely follows a proof in the cone types section of \cite{BH}. If $uwa$ is not a geodesic word, then there must exist some geodesic word $\ell$ of length strictly less than $|u|+|w|+1$ such that $\overline{\ell}=\overline{uwa}$. Write $\ell$ as a product $\ell_1\ell_2$ such that $|\ell_1|=|u|-1=|\overline{u}|_G-1$ and $|\ell_2| \leq |w|+1.$ Note that as $w \in \Pone^{D}(u)$, the word $uw$ is $D$-super-contracting and since $uw$ and $\ell$ end 1 apart from each other, Lemma \ref{lem: ending close to contracting implies fellow travelling} gives us that $d(\overline{\ell}_1,\overline{u})<C+1$. Define $z:=u^{-1}\ell_1$, hence, the group element $\overline{z}=\overline{u^{-1}\ell_1}$ satisfies $|\overline{z}|_G \leq C+1 \leq m$ and $|\overline{uz}|_G <|\overline{u}|_G$ which implies that $\overline{z} \in T_m(u).$ Recall that $T_m(u)=T_m(v)$ by assumption, so $\overline{z} \in T_m (v)$ and $|\overline{vz}|_G<|\overline{v}|_G$. Let $\alpha$ be any geodesic word connecting the identity to the group element $\overline{vz}$, so $|\alpha|<|\overline{v}|_G$. Now consider the concatenation of the geodesic $\alpha$ with the edge path labeled $\ell_2.$ On one hand, you get $\overline{\alpha \ell_2}=\overline{vz\ell_2}=\overline{vu^{-1}\ell_1\ell_2}=\overline{vu^{-1}uwa}=\overline{vwa}.$ Therefore, the edge path $\alpha \ell_2$ ends at the same vertex as the geodesic word $vwa$. Consequently, since $vwa$ is a geodesic, we have $|\alpha \ell_2| \geq |v|+|w|+1.$ But on the other hand, if we concatenate the geodesic word $\alpha$ with the edge path labeled $\ell_2$ we get $|\alpha \ell_2| \leq |\alpha|+|\ell_2| < |v|+|w|+1$ which is a contradiction. Therefore, $uwa$ must be a geodesic word.

\begin{figure}
    \centering

    \begin{tikzpicture}[scale=.3]
        
        \coordinate (id) at (0,0);
        \coordinate (u) at (-6,5);
        \coordinate (v) at (8, 3);
        \coordinate (w) at (-1,10);
        \coordinate (a) at (1,1);
        \coordinate (uw) at (-7,15);
        \coordinate (vw) at (7,13);
        \coordinate (uwa) at (-6,14);
        \coordinate (vwa) at (8,14);
        \coordinate (uz) at (-5,3);
        \coordinate (vz) at (7,2);

        \draw (id) node[vertex,label={[shift={(0,-.5)}]$id$}]{};
        \draw (u) node[vertex,label={[shift={(-0.2,0)}]$u$}]{};
        \draw (v) node[vertex,label={[shift={(.2,0)}]$v$}]{};
        \draw (uw) node[vertex, label={[shift={(0,.1)}]$uw$}]{};
        \draw (uwa) node[vertex, label={[shift={(.45,-.1)}]$uwa$}]{};
        \draw (vw) node[vertex, label={[shift={(-.3,0)}]$vw$}]{};
        \draw (vwa) node[vertex, label={[shift={(0,0.1)}]$vwa$}]{};
        \draw (uz) node[vertex, label={[shift={(-.3,-.3)}]$uz$}]{};
        \draw (vz) node[vertex, label={[shift={(0,-.4)}]$vz$}]{};
        \draw (-5,9) node[label={\color{blue}$\ell_2$}]{};
        \draw (8.6,9) node[label={\color{blue}$\ell_2$}]{};
        \draw (4.1,-.25) node[label={$\alpha$}]{};
        \draw (-3,-.1) node[label={$\ell_1$}]{};

        \draw (id)--(u);
        \draw (u)--(uw);
        \draw (uw)--(uwa);
        \draw (id)--(v);
        \draw (v)--(vw);
        \draw (vw)--(vwa);
        \draw (id) parabola (uz);
        \draw[very thick, color=blue] (uz) parabola (uwa);
        \draw[very thick, color=blue] (vz) parabola (vwa);
        \draw (id) parabola (vz);

        \draw (u) circle (3cm);
        \draw (v) circle (3cm);

    \end{tikzpicture}
     \caption{Geodesic cone is determined locally.}
    \label{fig: geodesic cone}
\end{figure}

Now we consider the other possibility, i.e. if $uwa$ is a geodesic word that is not $D$-super-contracting. Denote this geodesic by $\beta_1$, its geodesic subsegments labeled $wa$, $u$ by $\sigma_1$, $\sigma'_1$ respectively. Similarly, denote the geodesic $vwa$ by $\beta_2$, its geodesic subsegments labeled $wa$, $v$ by $\sigma_2$, $\sigma'_2$ respectively. Notice that if we let $g:=\overline{vu^{-1}}$, then we have $\sigma_2=g \cdot \sigma_1.$ So the assumption is that $\beta_1$ is not a $D$-super-contracting geodesic but $\beta_2$ is. This implies the existence of a subsegment $\gamma$ of $\beta_1$ that is not $D$-contracting. In other words, there exists a ball $B$ disjoint from $\gamma$, points $x,y \in B$,  and projection points $p_x \in \pi_{\gamma}(x),\, p_y \in \pi_{\gamma}
(y)$ such that $d(p_x,p_y)>D.$ 

Note that since the geodesic $uw$ is assumed to be $D$-super-contracting, then at least one of the points $p_x$ and $p_y$, say $p_x$, is on the edge labeled $a$ at the end of the geodesic $\beta_1.$ Moreover, we have $p_x \in \pi_{\sigma_1}(x).$

The first thing to observe is that, since $g$ is an isometry taking $\sigma_1$ to $\sigma_2$, $g \cdot p_x \in \pi_{\sigma_2}(g \cdot x)$ which is at the edge labeled $a$ at the end of $\sigma_2$.

\begin{figure}
    \centering

    \begin{tikzpicture}[scale=.3]
        
        \coordinate (id) at (0,0);
        \coordinate (u) at (-6,5);
        \coordinate (v) at (8, 3);
        \coordinate (w) at (-1,10);
        \coordinate (a) at (1,1);
        \coordinate (uw) at (-7,15);
        \coordinate (vw) at (7,13);
        \coordinate (uwa) at (-6,16);
        \coordinate (vwa) at (8,14);
        \coordinate (x) at (0, 15);
        \coordinate (y) at (0, 9);
        \coordinate (px) at (-6.5,15.5);
        \coordinate (py) at (-6.4,9);
        \coordinate (gx) at (14, 13);
        \coordinate (gy) at (14, 7);
        \coordinate (gpx) at (7.5, 13.5);
        \coordinate (gpy) at (7.6, 7);

        \draw (id) node[vertex,label={[shift={(0,-.5)}]$id$}]{};
        \draw (u) node[vertex,label={[shift={(-0.2,0)}]$u$}]{};
        \draw (v) node[vertex,label={[shift={(.2,0)}]$v$}]{};
        \draw (uw) node[vertex, label={[shift={(-0.4,-0.2)}]$uw$}]{};
        \draw (uwa) node[vertex, label={[shift={(.2,.1)}]$uwa$}]{};
        \draw (vw) node[vertex, label={[shift={(-0.35,-0.2)}]$vw$}]{};
        \draw (vwa) node[vertex, label={[shift={(0,.1)}]$vwa$}]{};
        \draw (x) node[vertex, label={[shift={(0,.05)}]$x$}]{};
        \draw (y) node[vertex, label={$y$}]{};
        \draw (px) node[vertex, label={[shift={(-0.2,0)}]$p_x$}]{};
        \draw (py) node[vertex, label={[shift={(-0.3,-0.2)}]$p_y$}]{};
        \draw (gx) node[vertex, label={$g\cdot x$}]{};
        \draw (gy) node[vertex, label={$g\cdot y$}]{};
        \draw (gpx) node[vertex, label={[shift={(-0.35,-0.05)}]$g\cdot p_x$}]{};
        \draw (gpy) node[vertex, label={[shift={(-0.35,-0.3)}]$g \cdot p_y$}]{};

        \draw (-7.3,10) node[label={\color{blue}$\sigma_1$}]{};
        \draw (8.5,8) node[label={\color{blue}$\sigma_2$}]{};
        \draw (-4.5,1) node[label={\color{red}$\sigma'_1$}]{};  \draw (4.5,.2) node[label={\color{red}$\sigma'_2$}]{};

        \draw [very thick, color=red](id)--(u);
        \draw[very thick, color=blue] (u)--(uw);
        \draw[very thick, color=blue] (uw)--(uwa);
        \draw [very thick, color=red](id)--(v);

        \draw[very thick, color=blue] (v)--(vw);
        \draw[very thick, color=blue] (vw)--(vwa);
        \draw[dotted] (x)--(px);
        \draw[dotted] (y)--(py);
        \draw[dotted] (gx)--(gpx);
        \draw[dotted] (gy)--(gpy);

    \end{tikzpicture}
     \caption{Contracting cone is determined locally.}
    \label{fig: contracting cone}
\end{figure}

Now we consider the two different possibilities for projections of $y$ to the subsegment $\gamma$ of $\beta_1$. If all projections of $y$ on the $\gamma$-subsegment of $\beta_1$ are on $\sigma_1$, then by our assumption, $\exists p_y \in \pi_{\sigma_1 \cap \gamma}(y)$ with $d(p_x,p_y)>D$. Let $\gamma'$ denote the subsegment of $\gamma$ which connects $p_x$ to $p_y$. Notice that since the ball $B$ containing $x,y$ is disjoint from $\gamma$, it must also be disjoint from $\gamma'$. By definition of $\gamma'$, we have $\gamma' \subseteq \sigma_1$ and since $g$ is an isometry taking $\sigma_1$ to $\sigma_2$, it must take $\gamma'\subseteq \sigma_1$ to $g \cdot \gamma' \subseteq \sigma_2$. Furthermore, since $p_x \in \pi_{\gamma'}(x), p_y \in \pi_{\gamma'}(y)$, we have

$$g\cdot p_x \in \pi_{g\cdot \gamma'}(g\cdot x),\, g\cdot p_y \in \pi_{g\cdot \gamma'}(g\cdot y),$$
\begin{center}

$g\cdot B \cap g\cdot \gamma' = \emptyset$, and $d(g\cdot p_x,g\cdot p_y)=d(p_x,p_y)>D,$
\end{center}
which contradicts the assumption that $\beta_2$ is $D$-super-contracting. The other possibility is that some projection of $y$ on the $\gamma$-subsegment of $\beta_1$ meets $\sigma'_1$. In other words, assume there exists a point $p \in \pi_\gamma(y) \cap \sigma_1'$. If that is the case, then by removing $\sigma'_1$ from $\beta_1$, using Lemma \ref{lemma: bounded jumps even wthout full contraction}, we get a point $p_y \in \pi_{\sigma_1}(y)$ which is at most $D$ away from $\overline{u}$. This implies that $d(p_x,p_y) \geq  (|\sigma_1|-1)-D= |w|-D \geq m-D >D$. Since $\sigma_1\subseteq \gamma,$ and $B \cap \gamma =\emptyset,$ we have $B\cap \sigma_1=\emptyset$. Again, as $g$ is an isometry with $g \cdot \sigma_1=\sigma_2$, we have 

\begin{center}
    
$g\cdot p_x \in \pi_{\sigma_2}(g\cdot x), \,g\cdot p_y \in \pi_{\sigma_2}(g\cdot y),$
\end{center}

\begin{center}

$g\cdot B \cap g\cdot \sigma_1=g\cdot B\cap \sigma_2=\emptyset, \text{ and }d(g\cdot p_x,g\cdot p_y)=d(p_x,p_y)>D.$
\end{center}
Hence, there exists a ball $B'=g\cdot B$ disjoint from $\sigma_2=g\cdot \sigma_1$ containing two points $g\cdot x,\,g\cdot y$ whose projections to $\sigma_2$ include points $g\cdot p_x,g\cdot p_y$ with $d(g\cdot p_x,g\cdot p_y)>D$. This contradicts the assumption that $\beta_2$ is $D-$super-contracting.

\end{proof}

\begin{definition}[acylindrical action] The action of a group $G$ on a metric space $X$ is called \emph{acylindrical} if for every $\epsilon >0$, there exist  $R,\, N>0$ such that for every two points $x, y \in X$ with $d(x,y)> R$, there are at most $N$ elements $g \in G$ satisfying:

$$d(x,gx)< \epsilon \,\,\, \text{     and   }\,\,\,d(y,gy)<\epsilon.$$
 
\end{definition}

\begin{definition}
 A group $G$ is said to be \emph{acylindrically hyperbolic} if it admits a non-elementary acylindrical action on a hyperbolic space.
\end{definition}

As a consequence of the above theorem, we will show that any finitely generated group which contains an infinite super-contracting geodesic must be acylindrically hyperbolic answering a question posed by Osin, but first, we state an easy corollary of Theorems H and I in \cite{Bestvina2014}:

\begin{corollary}
Let $G$ be a finitely generated group which is not virtually cyclic, and let $A$ be a generating set for $G$. If $G$ contains an element $g$ with a contracting axis in Cay$(G,A),$ then $G$ must be acylindrically hyperbolic.
\end{corollary}

Now we show how Theorem \ref{thm: regular language general} along with the previous corollary imply the following:

\begin{corollary}\label{restrictive}
Let $G$ be a finitely generated group which is not virtually cyclic, and let $A$ be a generating set for $G$. If Cay$(G,A)$ contains an infinite super-contracting geodesic, then it must contain a contracting isometry and hence $G$ must be acylindrically hyperbolic. In particular, $G$ can't be a torsion group.
\end{corollary}

\begin{proof}
Let $\gamma$ be an infinite $D$-super-contracting geodesic. Using Theorem \ref{thm: regular language general}, there exists a finite state automaton $M$ that accepts every initial subsegment of $\gamma$. Denote its accepted language by $L_D$. Notice that since $\gamma$ is an infinite $D$-super-contracting geodesic, the language $L_D$ must be infinite. Choose a large enough initial subsegment $w$ of $\gamma$ whose length is larger than the number of states in $M$. This implies the existence of a state that $w$ passes twice. In other words, $w$ has an initial subsegment of the form $uv$ such that $uv^n \in L_D$ for all $n \in \mathbb{N}$ (this is a standard argument in formal languages which is often referred to as the \emph{pumping lemma}). Since subsegments of $D$-super-contracting geodesics are themselves $D$-super-contracting, the subsegment $v^n$ of $uv^n$ is $D$-super-contracting. Hence, $\overline{v}$ is an element whose axis is super-contracting, and therefore, using the previous corollary, the group $G$ must be acylindrically hyperbolic.
\end{proof}

\section{Contracting is also super-Contracting} \label{sec: super contracting is equivalent to contracting}

Throughout this section, let $X$ denote a proper geodesic metric space. The goal of this section is to prove the following.
\begin{theorem} \label{thm: super-contracting is equivalent to contracting }
	Let $\alpha$ be a $D$-contracting geodesic in $X$. There exists a $D'$, depending only on $D$ so that $\alpha$ is $D'$-super-contracting. In particular, a geodesic $\alpha$ is $D$-contracting if and only if it is $D'$-super-contracting where $D$ and $D'$ determine each other.
\end{theorem}

To show this, we will use an additional hyperbolicity property, slimness, which we show in Lemma \ref{lem: slim} holds for contracting geodesics.

\begin{definition}\label{def: slim}
 A geodesic $\alpha$ is \emph{$\delta$-slim} if for any $x\in X$, $p\in\pi_\alpha(x)$, and geodesic $\beta$ from $x$ to a point on $\alpha$, we have $d(p, \beta)\leq\delta$.
\end{definition}

Contracting geodesics are known to also be slim in the setting of CAT(0) spaces. This property also holds in the generality of proper geodesic metric spaces, but the proof has to be adjusted slightly to accomodate the greater generality. In their proof that contracting implies slim, Bestvina and Fujiwara use the assumption that projection onto geodesics is coarsely distance decreasing \cite{BF2009}. This does not always hold for projections in a proper geodesic metric space, but it does hold for projections onto contracting geodesics, as we show in the following proposition.

\begin{proposition}\label{prop: distance decreasing}
    Let $\alpha$ be a $D$-contracting geodesic. For any points $x_1,x_2\in X$ and projection points $p_1\in\pi_\alpha(x_1), p_2\in\pi_\alpha(x_2)$,
    \[d(p_1,p_2) \leq d(x_1,x_2) + 4D. \]
\end{proposition}
\begin{proof}
    First we consider the case of two points $x$ and $y$ connected by a geodesic $[x,y]$ such that $d(w, \alpha) > D$ for all $w \in [x,y].$ Subdivide $[x,y]$ into $n=\left\lceil\frac{d(x,y)}{D}\right\rceil$ subsegments $[z_i,z_{i+1}]$ of length $\leq D$ with $z_0=x$ and $z_n=y$. So $nD \leq d(x,y)+D$. Let $z_0'\in \pi_\alpha(x)$ and $z_n'\in\pi_\alpha(y)$ be any pair of projection points. Pick $z_i'\in\pi_\alpha(z_i)$ for the rest. Since $\alpha$ is $D$-contracting, $d(z_i',z_{i+1}')\leq D$ for all $i$. Then \[d(z_0',z_n') \leq \sum_{i=0}^{n-1} d(z_i',z_{i+1}') \leq nD \leq d(x,y)+D. \]
    Now for the general case, fix any geodesic $[x_1,x_2]$. Let $y_1\in[x_1,x_2]$ be the point closest to $x_1$ satisfying $d(y_1,\alpha)\leq D$ and $y_2$ the analogous point closest to $x_2$. If no such points exist, we are in the previous case and already done. Pick any two projection points $q_i\in\pi_\alpha(y_i)$. By the triangle inequality, $d(q_1,q_2)\leq d(y_1,y_2)+2D$. The segments $[x_i,y_i]$ are at distance greater than $D$ from $\alpha$. Therefore we can apply the result of the previous paragraph to get
    \begin{align*}
        d(p_1,p_2) &\leq d(p_1,q_1) + d(q_1,q_2) + d(q_2,p_2) \\
        &\leq d(x_1,y_1)+D + d(y_1,y_2)+2D + d(y_2,x_2)+D \\
        &= d(x_1,x_2)+4D.
    \end{align*}
\end{proof}

\begin{lemma}\label{lem: slim}
	Let $\alpha$ be a $D$-contracting geodesic in a proper geodesic metric space. Then $\alpha$ is also $\delta$-slim where $\delta$ depends only on $D$.
\end{lemma} 
\begin{proof}
	Proposition \ref{prop: distance decreasing} shows that projections onto $D$-contracting geodesics satisfy the distance decreasing axiom of \cite{BF2009} with $C=4D$. The proof of slimness is then the same as in Lemmas 3.5 and 3.6 of \cite{BF2009}. From this proof, $\delta=7D+1$ suffices.
\end{proof}

\begin{remark}\label{rk: moving away from proj}
	Let $\alpha$ be a $D$-contracting (and thus $\delta$-slim by lemma \ref{lem: slim}) geodesic.  Let $x$ be any point in $X$ and $p \in \pi_\alpha(x)$. If $p_t$ is a point on $\alpha$ at distance $t$ away from $p$, and $z$ is a point in some geodesic connecting $x$ to $p_t$ with $d(p,z) \leq \delta$, then $d(x,p_t)=d(x,z)+d(z,p_t) \geq (d(x,p)-\delta)+(t-\delta).$ Therefore, we have
	\[ d(x, p_t) \geq d(x,p) + t - 2\delta. \]
\end{remark}
The following lemma looks very similar to Lemma \ref{lem: Gap}. The main difference is that here we only assume $\gamma$ is $D$-contracting, whereas Lemma \ref{lem: Gap} assumes that $\gamma$ is $D$-super-contracting.

\begin{lemma}\label{lem: cut subsegment}
		Let $\gamma:[a,b]\to X$ be a $D$-contracting geodesic and $\sigma$ any subsegment of $\gamma$. Let $\delta = 7D+1$. Suppose we have a point $x$ and a projection point $q\in\pi_\gamma(x)$ so that $q\notin\sigma$. If $u$ is the endpoint of $\sigma$ closest to $q$, then for any $p\in\pi_\sigma(x)$, $d(p,u)\leq 2\delta$.
\end{lemma}
\begin{proof}
	By Lemma \ref{lem: slim}, $\gamma$ is $\delta$-slim. 
	Let $p\in \pi_{\sigma}(x)$. The point $u$ is between $q$ and $p$ on a geodesic, so $d(q,p) = d(p,u) + d(u,q)$, and $d(x,p)\leq d(x,u)$ because $p$ is in the projection of $x$ onto $\sigma$. By the triangle inequality, $d(x,u) \leq d(u,q) + d(x,q)$. Starting from the inequality in Remark \ref{rk: moving away from proj} with $t=d(p,q)$, we have
	\begin{align*}
	    d(x,q) + d(q,p) &\leq d(x,p) + 2\delta \\
	    d(x,q) + d(p,u) + d(u,q) &\leq d(x,u) + 2\delta \\
	    d(x,q) + d(p,u) + d(u,q) &\leq d(u,q) + d(x,q) + 2\delta\\
	    d(p,u) &\leq 2\delta.
	\end{align*}
	\end{proof}

We are now ready to prove that subsegments of a contracting geodesic are also contracting. Our proof is essentially the same as that of Lemma 3.2 in \cite{BF2009}. We include it only to verify that is still valid assuming only that $X$ is a proper geodesic metric space, using Proposition \ref{prop: distance decreasing} in place of their distance decreasing axiom.

\begin{lemma}\label{lem: subsegments are contracting}
    Let $\gamma$ be a $D$-contracting geodesic. There is a constant $D'$, depending only on $D$, so that any subsegment $\sigma$ of $\gamma$ is $D'$-contracting. Consequently, $\gamma$ is $D'$-super-contracting.
\end{lemma}
\begin{proof}
    Let $\delta=\delta(D)$ be the slimness constant from Lemma \ref{lem: slim}. Let $u$ and $v$ be the endpoints of $\sigma$. In other words, $\sigma=[u,v]$ where $[u,v]$ denotes the the geodesic subsegment of $\gamma$ connecting $u$ to $v$. Let $a,b$ be the first and last points of $\gamma$ respectively.
	
	We claim that $D'=26D+4\delta+6$ does the job. We may assume that $\text{length}(\sigma) > D'$ as otherwise, there would be nothing to prove. Let $B$ be a ball centered at a point $z$ and disjoint from our subsegment $\sigma$. From here we divide the proof into two cases. The constant $D'$ is sufficient for both.

	{\bf Case 1.}
	$\pi_\gamma(z)\cap \sigma \neq\emptyset$. In this case, the ball $B$ is also disjoint from $\gamma$, so we know that the diameter of $\pi_\gamma(B)$ is at most $D$. Let $I=[s_1,s_2]$ denote the smallest interval containing all of the points in $\pi_\gamma(B)$. The length of $I$ is bounded above by $D$. We will argue that every point in $\pi_\sigma(B)$ is within a distance $2\delta$ of $I\cap\sigma$.

	If $x\in B$ has projection $\pi_{\gamma}(x)\cap\sigma\neq \emptyset$, then $\pi_\sigma(x) = \pi_\gamma(x)\cap\sigma$. In particular, $\pi_\sigma(x)\subset I$. Suppose instead $\pi_\gamma(x)$ contains a point $q\notin\sigma$. For concreteness, suppose $u$ is the endpoint of $\sigma$ nearest to $q$ (otherwise apply the same argument with $v$ instead of $u$). Then $u\in I$ and $d(u,p)\leq 2\delta$ for any $p\in\pi_\sigma(x)$ by Lemma \ref{lem: cut subsegment}. So everything in $\pi_\sigma(B)$ is within $2\delta$ of $I$, an interval of length at most $D$. Therefore the diameter of $\pi_\sigma(B)$ is at most $D+4\delta$ for Case 1.

    \begin{figure}
    \centering

    \begin{tikzpicture}
    \draw[thin,-] (-2,0) -- (5,0);

    \draw[thick] (0,1) circle (1.2cm);
    \node[above] at (0,1.1) {$z$};
    \draw[thin, red, -] (1,0) -- (4,0);
    \node[below] at (1,-.1) {$u$};
    \draw[thick,fill=black] (1,0) circle (.02cm);

    \node[below] at (4,-.1) {$v$};
    \draw[thick,fill=black] (4,0) circle (.02cm);

    \node[below] at (-2,-.1) {$a$};
    \draw[thick,fill=black] (-2,0) circle (.02cm);

    \node[below] at (5,-.1) {$b$};
    \draw[thick,fill=black] (5,0) circle (.02cm);

    \node[below] at (2.5,-.1) {$\sigma$};

    \draw[thick,fill=black] (0,1) circle (.02cm);
    \draw[thick,fill=black] (1,1) circle (.02cm);
    \node[above] at (1,1.1) {$x$};

    \draw[thin,-] (0,1) -- (1,1);

     \end{tikzpicture}
     \caption{Case 2 of Lemma \ref{lem: subsegments are contracting} }
    \label{fig: Case 2}
\end{figure}

	{\bf Case 2.} $\pi_\gamma(z)\cap \sigma= \emptyset$. Notice that in this case it is possible that $B$ intersects $\gamma$. Every point $p_{z} \in \pi_{\gamma}(z)$ must belong to precisely one of $[a,u)$, $(v,b]$. In other words, since $\gamma$ is $D$-contracting, it is not possible to have two points $p_z,p_z' \in \pi_{\gamma}(z)$ with $p_z \in [a,u)$ and $p_z' \in (v,b]$. Without loss of generality, suppose that every $p_z \in [a,u).$ We want to show, for all $x \in B$ and $p_x \in \pi_{\sigma}(x)$, that $d(p_x,u) \leq 13D+2\delta + 3$. This would imply $d(p_x,p_y) \leq 26D+4\delta + 6=D'$ for all $x,y \in B$ and $p_x ,p_y \in \pi_{\sigma}(x), \pi_{\sigma}(y)$ respectively. 
	
	To show the above, assume for the sake of contradiction that there exists $x\in B$ so that $\pi_{\sigma}(x)$ contains a point further than $13D+2\delta + 3$ from $u$. Notice that no projection point of $x$ to $\gamma$ can be on $[a,u)$ because Lemma \ref{lem: cut subsegment} would imply $d(\pi_\sigma(x),u) \leq 2 \delta$. So there is also at least one point $p_x\in\pi_\gamma(x)$ with $d(p_x,u) > 13D+2\delta + 3$. Fix $[z,x]$ to be some geodesic connecting $z$ to $x$. By subdividing $[z,x]$ into intervals of length at most one, and using Proposition \ref{prop: distance decreasing}, we can find a point $y \in [z,x]$ such that the set $\pi_{\gamma}(y)$ is contained in the $1+4D$ neighborhood of $u.$ Let $B'$ be the closed ball centered at $y$ with radius $r=d(y,\gamma)-1$ and since $\gamma$ is $D$-contracting, the set $\pi_{\gamma}(B')$ has diameter $\leq D.$ Notice that 
	\[
	d(y,u) \geq d(z,u)-d(z,y) = d(z,u)-d(z,x)+d(y,x) \geq d(y,x).
	\]
	Now, since the set $\pi_{\gamma}(y)$ is contained in the $1+4D$ neighborhood of $u$, we have
	\[d(y,x) \leq d(y,u) \leq d(y,\gamma) + 4D+1 = r+4D+2\]
	Therefore, by the previous inequality, the ball $B'$ must contain a point $w$ with $d(w,x) \leq 4D+2.$ Now notice that for any $p_y \in \pi_{\gamma}(y)$ and $p_w\in\pi_\gamma(w)$, we have
	\[d(p_y,p_x) \leq d(p_y,p_w)+d(p_w,p_x) \leq D+((4D+2)+4D)=9D+2.\]
	On the other hand, we have $d(u,p_y) \leq 1+4D$, which gives that $d(u,p_x) \leq d(u,p_y)+d(p_y,p_x) \leq (1+4D)+(9D+2)=13D+3$, this is a contradiction since $d(u,p_x)>13D+2\delta+3.$

\end{proof}

In light of Theorem \ref{thm: super-contracting is equivalent to contracting }, we can apply Corollary \ref{restrictive} to show the following.
\begin{corollary}
Let $G$ be a finitely generated group which is not virtually cyclic, and let $A$ be a generating set for $G$. If Cay$(G,A)$ contains an infinite contracting geodesic, then $G$ must be acylindrically hyperbolic.
\end{corollary}

\begin{remark}\label{rmk: contraction characterization for hyperbolicity} In the introduction, we stated that a geodesic metric space is hyperbolic if and only if there exists a constant $D$ such that every geodesic is $D$-super-contracting. The forward direction is a standard fact about hyperbolic spaces. For example, the argument given in Theorem 2.14 of \cite{ChSu2014} proving that for a CAT(0) space $X$ every $\delta$-slim geodesic is $D$-contracting (and hence $D'$-super-contracting) still works when the CAT(0) space $X$ is replaced by a hyperbolic space. Conversely, if every geodesic is $D$-super-contracting, then by Lemma \ref{lem: contracting implies Morse}, there exists an $M$ such that every geodesic is $M$-Morse. Now, Lemma 2.2 of \cite{Cordes2017} states that if two edges of a triangle in a geodesic metric space are $M$-Morse, then the triangle is $4M(3,0)-$thin. This implies that the space is $4M(3,0)$-hyperbolic.

\end{remark}

\bibliography{bibliography}{}
\bibliographystyle{plain}

\end{document}